\documentclass[11pt]{amsart}
\usepackage{graphicx}
\usepackage[usenames]{color}
\usepackage{amsfonts,mathrsfs}
\usepackage{amssymb,amscd}
\usepackage{graphicx,tikz}
\usetikzlibrary{arrows,automata,chains,fit,shapes}
\newcommand\N{\mathbb N}

\newcommand\F{\mathcal{F}}

\newcommand\conv{\otimes}
\newcommand{\ou}{\overline{u}}
\newcommand{\ov}{\overline{v}}
\newcommand\ST{\Sigma_T}
\newcommand\LT{L_{tree}}
\newcommand\LTT{L_T}
\newcommand\SI{\Sigma_{int}}
\newcommand\LI{L_{int}}

\newtheorem{theorem}{Theorem}[section]
\newtheorem{proposition}[theorem]{Proposition}
\newtheorem{lemma}[theorem]{Lemma}

\newtheorem{defn}[theorem]{Definition}

\title{Tree-based Language complexity of Thompson's group F}
\author{Jennifer Taback and Sharif Younes}
\address{Department of Mathematics,
Bowdoin College, Brunswick, ME 04011} \email{jtaback@bowdoin.edu}
\address{NuFit Media, 1 Kendall Square, Suite B2101, Cambridge, MA 02139} \email{sharif@nuftmedia.com}
\thanks{The first author acknowledges support from National Science Foundation grant DMS-1105407 and Simons Foundation grant 31736 to Bowdoin College. Both authors wish to thank Sean Cleary, Bob Gilman, Alexei Miasnikov and especially Murray Elder for many helpful discussions during the writing of this paper.}

\begin{document}

\maketitle

\begin{abstract}
The definition of graph automatic groups by Kharlampovich, Khoussainov and Miasnikov and its extension to ${\mathcal C}$-graph automatic by Murray Elder and the first author raise the question of whether Thompson's group $F$ is graph automatic.  We define a language of normal forms based on the combinatorial ``caret types" which arise when elements of $F$ are considered as pairs of finite rooted binary trees, which we show to be accepted by a finite state machine with $2$ counters, and forms the basis of a $3$-counter graph automatic structure for the group.
\end{abstract}

\section{Introduction}\label{sec:intro}

It is not known whether Thompson's group $F$ is automatic; $F$ has some characteristics of an automatic group such as quadratic Dehn function, type $FP_{\infty}$ and word problem solvable in $O(n \log n)$ time (see \cite{GS}, \cite{BG} and \cite{SU}, respectively).  Yet an automatic structure for the group remains elusive.  Guba and Sapir present a regular language of normal forms for elements of $F$ in \cite{GS}, and Cleary and the first author show that $F$ has no regular language of geodesics in \cite{CT}.

With the extension of the notion of an automatic group to a graph automatic group by Kharlampovich, Khoussainov and Miasnikov in \cite{KKM}, and further to a ${\mathcal C}$-graph automatic group by Elder and the first author in \cite{ET}, a natural question is whether $F$ is captured by one of these larger classes of groups.  To summarize these definitions for a group $G$ with finite generating set $S$:

\medskip

\begin{enumerate}
\item $G$ has an automatic structure if there is a normal form for group elements that is recognized by a finite state machine, as well as finite state machines $M_s$ that accept the language of pairs $(g,gs)$, for each $s \in S$.  The machines $M_s$ are called the multiplier automata.
\item $G$ has a graph automatic structure if there is a finite symbol alphabet used to define a normal form for group elements that is recognized by a finite state machine, as well as finite state machines $M_s$ that accept pairs of normal form words corresponding to group elements that differ by $s$, for each $s \in S$.
\item $G$ has a ${\mathcal C}$-graph automatic structure, where ${\mathcal C}$ is a class of languages---for example regular, context free, counter or context-sensitive---and all languages in the previous definition are now allowed to lie in the class ${\mathcal C}$.
\end{enumerate}

The introduction of a symbol alphabet used to express normal forms for group elements seems quite suited to Thompson's group $F$.  If group elements are expressed as reduced pairs of finite rooted binary trees, there is combinatorial information that can be used to express the group element uniquely.  Namely, these trees are constructed from nodes, or {\em carets}, and there are several descriptions of caret types given in the literature that depend on the placement of the carets within the tree.  We follow those given by Elder, Fusy and Rechnitzer in \cite{Murray} and define a quasi-geodesic normal form for elements of $F$.

We show below that $F$ is not graph automatic with respect to this natural normal form language over this set of symbols. This does not rule out the result with respect to a different normal form, or an alternate symbol alphabet.  However, any symbol alphabet which captures the tree structure in any way seems unlikely to the authors to yield the result that $F$ is graph automatic.

In \cite{ET1}, Elder and the first author show that the standard infinite normal form for elements of $F$ is the basis of a $1$-counter graph automatic structure.  The geodesic normal form we construct below based on caret types yields a $3$-counter graph automatic structure for the group.  This could be improved to a $2$-counter graph automatic structure, but the increase in complexity of the argument did not balance out the decrease in counters, given the result in \cite{ET1}.

We remark that a third possibility for constructing a $\mathcal C$-graph automatic structure for $F$ would be to use the regular normal form language of Guba and Sapir \cite{GS}.  In preparing this article we considered the complexity of the multiplier automata for this language and found that the $x_1$ multiplier automaton would require at least $4$ counters, and hence the resulting structure is less efficient overall than the structures obtained in \cite{ET1} and this paper.

The paper is organized as follows.  In Section \ref{sec:graph-aut} we recall the definition of ${\mathcal C}$-graph automatic groups introduced in \cite{ET} which generalizes the definition of graph automatic groups introduced by Kharlampovich, Khoussainov and Miasnikov in \cite{KKM}.  In Section \ref{sec:F} we present relevant background on Thompson's group $F$.  In Sections \ref{Sec:caret-lang} and \ref{subsec:x_0multiplier} we show that the language based on the symbol alphabet consisting of caret types forms a $3$-counter graph automatic structure for $F$.

\section{Background on Generalizations of Automaticity}\label{sec:graph-aut}

There are many definitions in the literature of counter automata; we begin with our definition.
\begin{defn}[counter automaton]
A {\em nonblind deterministic $k$-counter automaton} is
  a deterministic finite
state automaton augmented with {\em k} integer counters: these are all
initialized to zero, and can be incremented, decremented, compared to zero and set to zero during operation.  For each configuration of the machine and subsequent input letter, there is at most one possible move. The automaton accepts a word exactly if upon reading the word it  reaches
an accepting state with all counters returned to zero.
\end{defn}

In drawing a counter automaton, we label transitions by the input letter to be read, with subscript to denote the possible counter instructions:

\medskip

\begin{itemize}
\item $=0$ to indicate the edge may only be traversed if the value of the counter is $0$.
\item $+1$ to increment the counter by $1$.
\item $-1$ to decrement the counter by $1$.
\end{itemize}

In this paper we focus entirely on counter-graph automatic structures, which will be defined below.  The class of counter languages is closed  under homomorphism, inverse homomorphism, intersection with regular languages, and finite intersection (see \cite{HU}, for example).  Moreover, the intersection of a $k$-counter  language with a regular language is $k$-counter, and the intersection of $k$- and $l$-counter languages is a $(k+l)$-counter language, as proven in \cite{ET}.

We introduce the notion of a convolution of strings in a language, following \cite{KKM}, for ease of notation in the multiplier languages we later define.  Let $G$ be a group with symmetric  generating set $X$, and $\Lambda$ a finite set of symbols. In general we do not assume that $X$ is finite.
The number of symbols (letters) in a word $u\in\Lambda^*$ is denoted $|u|_{\Lambda}$.

\begin{defn}[convolution;  Definition 2.3 of \cite{KKM}]
Let $\Lambda$ be a finite set of symbols, $\diamond$  a symbol not in $\Lambda$,  and let $L_1,\dots, L_k$ be a finite set of languages over $\Lambda$. Set  $\Lambda_{\diamond}=\Lambda\cup\{\diamond\}$. Define the {\em convolution of a tuple} $(w_1,\dots, w_k)\in L_1\times \dots \times L_k$ to be the string $\otimes(w_1,\dots, w_k)$ of length $\max (|w_i|_{\Lambda})$ over the alphabet $\left(\Lambda_{\diamond}\right)^k$  as follows.
The $i$th symbol of the string is
\[\left(\begin{array}{c}
\lambda_1\\
\vdots\\
\lambda_k
\end{array}\right)\]
where $\lambda_j$ is the $i$th letter of $w_j$ if $i\leq |w_j|_{\Lambda}$ and $\diamond$ otherwise.
Then  \[\otimes(L_1,\dots, L_k)=\left\{\otimes(w_1,\dots, w_k) \mid w_i\in L_i\right\}.\]
\end{defn}

As an example, if $w_1=aa, w_2=bbb$ and $w_3=a$ then
\[\otimes(w_1,w_2,w_3)=\left(\begin{array}{c}
a\\
b\\
a
\end{array}\right)
\left(\begin{array}{c}
a\\
b\\
\diamond
\end{array}\right)
\left(\begin{array}{c}
\diamond\\
b\\
\diamond
\end{array}\right)\]

When $L_i=\Lambda^*$ for all $i$ the definition in \cite{KKM} is recovered.

We require that the normal form language on which a ${\mathcal C}$-graph automatic structure is based be quasigeodesic.
\begin{defn}[quasigeodesic normal form]
A  {\em normal form for $(G,X,\Lambda)$} is a set of words $L\subseteq \Lambda^*$ in bijection with $G$. A normal form $L$ is
  {\em quasigeodesic}  if there is a constant $D$ so  that $$|u|_{\Lambda}\leq D(||u||_X+1)$$ for each $u\in L$,  where $||u||_X$ is the length of a geodesic  in $X^*$ for the group element represented by  $u$.
\end{defn}
The $||u||_X+1$ in the definition allows for normal forms where the identity of the group is represented by a nonempty string of length at most $D$.  We denote the image of $u\in L$ under the bijection with $G$ by $\ou$.

We now state the definition of a ${\mathcal C}$-graph automatic group, following \cite{ET}, which generalizes the notion of a graph automatic group introduced in \cite{KKM}.

\begin{defn}[$\mathcal C$-graph automatic group] \label{def:Cgraph}
Let $\mathcal C$ be a formal language class, $(G,X)$ a group and symmetric  generating set, and $\Lambda$ a finite set of symbols.  We say that $(G,X,\Lambda)$ is $\mathcal C${\em -graph automatic} if there is a  normal form $L \subset \Lambda^*$ in the language class $\mathcal C$, such that for each $x\in X$ the language $L_x=\{\otimes(u,v) \mid u,v\in L,  \ov =_G \ou x\}$ is in the class $\mathcal C$.
\end{defn}

The following observation is proven in \cite{ET}.

\begin{lemma}[Lemma 2.5 of \cite{ET}]\label{lemma:convolution}
Let $L_1$ and $L_2$ be $k$- and $l$-counter languages respectively.  Then $\otimes(L_1,L_2)$ is a $(k+l)$-counter language.
\end{lemma}

We end this section with a lemma describing when certain languages of convolutions are regular; this lemma will be used to streamline many of the proofs in later sections.
\begin{lemma}\label{L:RegularLanguages}
Fix a symbol alphabet $\Lambda$ and let $a_1,a_2,b_1,b_2$ be fixed (possibly empty) words in $\Lambda^*$. Then the following languages are all regular.

\medskip

\begin{enumerate}
\item $L_1 = \{\conv(a_1w,a_1b_1w) \mid w \in \Lambda^*\}$
\item $L_2 = \{\otimes(a_1w,b_1w) \mid w \in \Lambda^*\}$ and $L_3 = \{\otimes(wa_2,wb_2) \mid w \in \Lambda^*\}$
\item $L_4 = \{\otimes(w_1a_1w_2a_2w_3,w_1b_1w_2b_2w_3 \mid \ w_1, w_2, w_3 \in \Lambda^*\}$
\end{enumerate}
\end{lemma}
\begin{proof}[Proof of (1)] Since $\{\otimes(a_1,a_1)\}$ is regular and the set of regular languages is closed under concatenation, this problem reduces to recognizing the language
$$L = \{\conv(w,bw)\}$$
with a finite state automaton.

Let $|b| = p$. Enumerate all strings of length $p$ in $\Lambda^*$: $s_1, s_2, s_3, \ldots , s_n$. For each $s_i$, create a state $S_i$ and add a path of $p$ edges from the start state $q_0$ to $S_i$ labeled by the successive letters in $\otimes(s_i, b)$. Next, for each pair $(\lambda_1, \lambda_2) \in \Lambda \times \Lambda$ add an edge from $S_i \to S_j$ labeled $(\lambda_1, \lambda_2)$ if and only if $s_i = \lambda_2 x$ and $s_j = x \lambda_1$ for $x \in \Lambda^*$. Finally, for each $s_i$, construct a path from $S_i$ to an accept state $F$ labeled by $\conv(\epsilon, s_i) = (\diamond^p, s_i)$.

The edge labels ensure that if a prefix of $\otimes(u,v)$ describes a path from the start state to state $S_i$, then the last $p$ letters read from $u$ are exactly the string $s_i$. Reading the next  letter of $\otimes(u,v)$, say ${x \choose y}$, to transition to state $S_j$ means that the final $p$ letters now read from $u$ are $s_j$ and that the initial letter of $s_i$ was $y$.  In this way we check that the string $bw$ contains the string $w$ shifted $p$ places by the insertion of $b$.

Cases (2) and (3) are proven by similar arguments and we omit them here.
\end{proof}

\section{Background on Thompson's group $F$}\label{sec:F}
Thompson's group $F$ can be equivalently viewed from three perspectives:

\medskip

\begin{enumerate}
\item as the group defined by the infinite presentation
$$ {\mathcal P}_{inf}=\langle x_0,x_1, x_2\cdots | x_jx_i = x_ix_{j+1} \text{ whenever } i<j \rangle $$
or the finite presentation
$$ {\mathcal P}_{fin}=\langle x_0,x_1| [x_0^{-1}x_1,x_0^{-1}x_1x_0],[x_0^{-1}x_1,x_0^{-2}x_1x_0^2] \rangle. $$
\item as the set of piecewise linear homeomorphisms of the interval $[0,1]$ satisfying
\begin{enumerate}
\item each homeomorphism has finitely many linear pieces,
\item all breakpoints have coordinates which are dyadic rationals, and
\item all slopes are powers of two.
\end{enumerate}

\medskip

\item as the set of pairs of finite rooted binary trees with the same number of nodes, or carets.
\end{enumerate}
For the equivalence of these three interpretations of this group, as well as a more complete introduction to the group, we refer the reader to \cite{CFP}. If $g=(T,S)$ is an element of $F$ given as a pair of finite rooted binary trees, then the tree $T$ determines the linear intervals in the domain of $g$, and the tree $S$ determines the linear intervals in the range; again see \cite{CFP} for details.

\subsection{Terminology and notation}
In this section we define the terminology that we use in creating the $2$-counter language of normal forms for elements of $F$.  A {\em caret} consists of a vertex in a tree $T$ together with two edges that we draw with a downward orientation.  A vertex in the tree $T$

\medskip

\begin{enumerate}
\item is called a {\em leaf} iff it has valence one;
\item is the {\em root} of the tree iff it has valence two; and
\item has valence three otherwise.
\end{enumerate}

A caret in $T$ has a left edge and a right edge; we refer to a caret $D$ as the {\em left (resp. right) child} of caret $C$ if the vertex of $D$ is the endpoint of the left (resp. right) edge of $C$.  A caret is called {\em exterior} if at least one of its edges lies on the left or right side of the tree, and {\em interior} otherwise.  A caret is {\em exposed} if it has no children.

Let $(T,S)$ denote a pair of finite rooted binary trees with the same number of carets, which necessarily corresponds to a unique element of Thompson's group $F$. We refer to this as a tree pair diagram representing the element. The carets of $T$ and $S$ are numbered in increasing infix order: beginning with the leftmost descendant of the root caret, the left child of any caret is numbered before the caret, and the right child is numbered after it.  We then pair carets with the same infix number from the two trees and refer to a {\em caret pair} with caret number $n$.

\subsection{Reduction criterion}
To ensure a bijective correspondence between elements of $F$ and pairs of finite rooted binary trees with the same number of carets, we impose an additional requirement.  Viewing an element of $F$ as a piecewise linear homeomorphism of the interval, subject to the above conditions, we can always trivially add additional breakpoints within a linear component without altering the function.  However, such breakpoints are superfluous. We define a {\em reduced} tree pair diagram to correspond to a function without unnecessary breakpoints.  Exposed carets with identical infix numbers in both trees introduce a superfluous breakpoint to the analytic representation of a group element.  By removing these unnecessary caret pairs we obtain a reduced tree pair diagram.
An element $g \in F$ is thus represented by an equivalence class of tree pair diagrams with a unique reduced diagram.  When we refer to a tree pair diagram representing a group element, we mean the reduced diagram unless otherwise indicated.

\subsection{Multiplication of tree pair diagrams}\label{sec:Fmult}
Multiplication of elements of $F$ when viewed as piecewise linear homeomorphisms of the interval $[0,1]$ is simply composition of functions; ``multiplication'' of tree pair diagrams is a combinatorial construction which mimics composition of functions.  Consider $g=(T_1,T_2)$ and $h=(S_1,S_2)$ in $F$; create unreduced representatives of both $g$ and $h$, which we denote by $(T'_1,T'_2)$ and $(S'_1,S'_2)$ respectively, so that $T'_2=S'_1$.  The product $hg$ is then represented by the possibly unreduced tree pair diagram $(T'_1,S'_2)$.  For explicit examples of multiplication of elements of $F$ we refer the reader to \cite{CT1}.

Multiplication by the generators $x_0$ and $x_1$ and their inverses induce particular combinatorial rearrangements of a given tree pair diagram which are explicitly described in Section \ref{subsec:x_0multiplier}.

\subsection{Caret labels} \label{sec:carettypes}
In Section \ref{Sec:caret-lang} below we develop a set of normal forms for elements of $F$ using an alphabet consisting of {\em caret types} or {\em labels}.  We give each caret a label based on its position in the tree, and form these strings into pairs by infix order.

There are several existing sets of caret labels in the literature; Fordham in \cite{F} uses caret labels to assign weights to each caret pair.  He then proves that the sum of the weights corresponding to a (reduced) element is exactly its word length with respect to the finite generating set $\{x_0,x_1\}$. Unfortunately, Fordham's caret labels are not in one-to-one correspondence with configurations of carets in a tree, as illustrated in Figure \ref{F:NotUnique}.  Fordham defines a caret to be {\em interior} if neither edge comprising the caret lies on the left or right side of the tree, and further subdivides these carets into type $I_0$, which have no right child, and type $I_R$, which do.  The configurations of carets in Figure \ref{F:NotUnique} both correspond to the string of labels $I_R,I_0,I_0$, where the labels are listed in infix order.

\begin{figure}[h!]
\begin{center}
\begin{tikzpicture}[scale=.65]
		\coordinate (A') at (-.5,1) {};
		\coordinate (A) at (0,0) {};
		\coordinate (B) at (-.5,-1) {};
		\coordinate (C) at (.5,-1) {};
		\coordinate (D) at (0, -2) {};
		\coordinate (E) at (-1, -2) {};
		\coordinate (F) at (-.5,-3) {};
		\coordinate (G) at (.5,-3) {};

		\node (X) at (0, -4) {$I_R, I_0, I_0$};
		\node (X1) at (0,-.75) {$I_0$};
		\node (X2) at (-.5,-1.75) {$I_R$};
		\node (X3) at (0,-2.75) {$I_0$};
			
		\draw (B) -- (A) -- (C);
		\draw (D) -- (B) -- (E);
		\draw (G) -- (D) -- (F);
		\draw[dashed] (A') -- (A);
		
		\coordinate (Q') at (3,1) {};
		\coordinate (Q) at (3.5,0) {};
		\coordinate (R) at (4,-1) {};
		\coordinate (S) at (3,-1) {};
		\coordinate (T) at (4.5, -2) {};
		\coordinate (U) at (3.5, -2) {};
		\coordinate (V) at (3,-3) {};
		\coordinate (W) at (4,-3) {};
		
		\node (Y) at (3.5,-4) {$I_R,I_0,I_0$};
		\node (Y1) at (3.5, -.75) {$I_R$};
		\node (Y2) at (4, -1.75) {$I_0$};
		\node (Y3) at (3.5,-2.75) {$I_0$};
		
		\draw (R) -- (Q) -- (S);
		\draw (T) -- (R) -- (U);
		\draw (V) -- (U) -- (W);
		\draw[dashed] (Q') -- (Q);

    \foreach \point in {A,B,C,D,E, F,G, Q,R,S,T,U,V,W}
        \fill [black] (\point) circle (1.5pt);

\end{tikzpicture}
\caption{Fordham's caret labels are not in bijective correspondence with binary trees: two different interior subtrees corresponding to the same string of labels.  Labels are listed in infix order.} \label{F:NotUnique}
\end{center}
\end{figure}
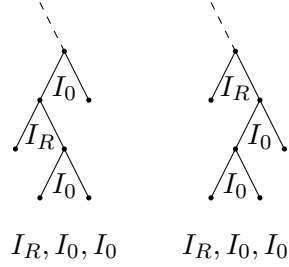 

We refine Fordham's definition of interior caret types to avoid the above problem, and condense his left and right caret types (comprising carets with an edge on either the left or right side of the tree) into a single exterior type, while distinguishing the root caret.  Define the following caret types within a finite rooted binary tree:

\medskip

\begin{enumerate}
\item the root caret is labeled $``r"$,
\item all other exterior carets are labeled $``e"$,
\item interior carets are labeled as follows:
\begin{table}[ht!]
\begin{center}
\begin{tabular}{|c|c|c|}
  \hline
   & is the left child of an  & is not the left child of an  \\
   Caret $C$ & interior caret or the & interior caret nor the  \\
   & child of an exterior caret & child of an exterior caret \\ \hline
  has a right child & ( & b \\ \hline
  has no right child & a & ) \\
  \hline
\end{tabular}
\end{center}
\end{table}
\end{enumerate}

Listing the caret types of the carets in a binary tree in infix order yields a word over the alphabet $\ST = \{r, e, (, ), a,b\}$.  For example, the tree in Figure~\ref{F:IntTree2} is encoded by the word $eea()(ab)raee$.

\begin{figure}[ht!]
\begin{center}
\begin{tikzpicture}
  [sibling distance = 14mm,
   level distance=16mm,
   level 1/.style={sibling distance=64mm},
   level 2/.style={sibling distance=32mm},
   level 3/.style={sibling distance=24mm},
   level 4/.style={sibling distance=12mm},
   level 5/.style={sibling distance=8mm},
   level 6/.style={sibling distance=4mm},
   every node/.style={circle,inner sep=1.5pt}]

\coordinate node [circle,draw] {$r$}
	child {node [circle,draw]{$e$}
		child{node [circle,draw]{$e$}
			child{node [circle,inner sep=.75pt,draw,fill] {}}
			child{node [circle,inner sep=.75pt,draw,fill] {}}
		}
		child{node [circle,draw]{$($}
			child{node [circle,inner sep = 1pt, draw]{$($}
				child {node [circle,inner sep = 1pt, draw]{$a$}
					child{node [circle,inner sep=.75pt,draw,fill] {}}
					child{node [circle,inner sep=.75pt,draw,fill] {}}
				}
				child	{node [circle,inner sep = 1pt, draw]{$)$}
					child{node [circle,inner sep=.75pt,draw,fill] {}}
					child{node [circle,inner sep=.75pt,draw,fill] {}}
				}
			}
			child{node [circle,inner sep = 1.25pt,draw]{$b$}
				child {node [circle,inner sep = 1pt, draw]{$a$}
					child{node [circle,inner sep=.75pt,draw,fill] {}}
					child{node [circle,inner sep=.75pt,draw,fill] {}}
				}
				child {node [circle,inner sep = 1pt, draw]{$)$}
					child{node [circle,inner sep=.75pt,draw,fill] {}}
					child{node [circle,inner sep=.75pt,draw,fill] {}}
				}
			}
		}
	}
	child {node [circle,draw]{$e$}
		child {node [circle,draw]{$a$}
			child{node [circle,inner sep=.75pt,draw,fill] {}}
			child{node [circle,inner sep=.75pt,draw,fill] {}}
		}
		child{node [circle,draw]{$e$}
			child{node [circle,inner sep=.75pt,draw,fill] {}}
			child{node [circle,inner sep=.75pt,draw,fill] {}}
		}
	}
	;
\end{tikzpicture} \linebreak
\setlength{\parindent}{.1 in}
\caption{An example of a tree corresponding to the string of labels $eea()(ab)raee$.} \label{F:IntTree2}
\end{center}
\end{figure} 

These caret labels are easily translated into those used by Elder, Fusy and Reichnitzer in \cite{Murray} (which are based on diagrams in \cite{BB});  their labels mark only interior subtrees. Leaves are labeled ``N'' or ``I'' and carets are labeled ``n'' or ``i'' as in Figure~\ref{F:Labels1}.

\begin{figure}[ht!]
\begin{center}
\begin{minipage}[b]{0.15\linewidth}
\centering
\begin{tikzpicture}
  [level distance=10mm, sibling distance = 14mm,
   every node/.style={circle,inner sep=1.5pt}]

\coordinate node [circle,fill,draw] {}
	child {node [circle,draw]{N}}
	child [missing]
	;
\end{tikzpicture}
\end{minipage}
\begin{minipage}[b]{0.15\linewidth}	
\centering	
	\begin{tikzpicture}
	  [level distance=10mm, sibling distance = 14mm,
   every node/.style={circle,inner sep=1.5pt}]

\coordinate node [circle,fill, draw] {}
	child [missing]
	child {node [circle,draw]{I}}
	;
\end{tikzpicture}
\end{minipage}
\begin{minipage}[b]{0.15\linewidth}	
\centering	
	\begin{tikzpicture}
	  [level distance=5mm, sibling distance = 7mm,
   every node/.style={circle,inner sep=1pt}]
\setlength{\parindent}{.14 in}

\indent \coordinate node [circle,fill, draw] {}
	child {node [circle, draw] {n}
		child {node [circle,fill, draw] {}}
		child {node [circle,fill, draw] {}}
	}
	child [missing]
	;
\end{tikzpicture}

\end{minipage}
\begin{minipage}[b]{0.15\linewidth}	
\setlength{\parindent}{.12 in}
 \indent \centering	
	\begin{tikzpicture}
	  [level distance=5mm, sibling distance = 7mm,
   every node/.style={circle,inner sep=1pt}]

\coordinate node [circle,fill, draw] {}
	child [missing]
	child {node [circle, draw] {i}
		child {node [circle,fill, draw] {}}
		child {node [circle,fill, draw] {}}
	}
	;
\end{tikzpicture}

\end{minipage}
\caption{Leaf and caret labels used in \cite{Murray}.  Leaves are labeled ``N" or ``I" and carets are labeled ``n" or ``i." By convention, the root caret is labeled ``n" and the left-most leaf is not labeled.} \label{F:Labels1}
\end{center}
\end{figure}

There is a straightforward translation between the notation in \cite{Murray} and our own:
\begin{enumerate}
\item If a caret is labeled ``n,'' the caret is a left child or else the root caret;
\item if the caret is labeled ``i,'' then the caret is a right child.
\item If the leaf directly to the right of a caret is labeled ``N,'' then the caret has a right child; if the leaf directly right of a caret is labeled ``I,'' then the caret does not.
\end{enumerate}
Following these rules, replace any appearance of ``nN'' with ``$($'', ``iI'' with ``$)$'', ``nI'' with ``$a$'', and ``iN'' with ``$b$''.

\section{A 2-counter language of normal forms for elements of $F$}
\label{Sec:caret-lang}

We now define a language $\F$ of normal forms for elements of $F$ based on the caret types defined in Section \ref{sec:carettypes} which is accepted by a nonblind deterministic $2$-counter automaton.  This forms the basis of a
$3$-counter graph automatic structure for $F$; the additional counter is required to construct the multiplier automaton for $x_1$.

\begin{theorem} \label{T:GCS}
Thompson's group $F$ is nonblind deterministic $3$-counter-graph automatic with quasigeodesic normal form with respect to the generating set $S = \{x_0, x_1\}$ and the symbol alphabet $\ST = \{e,r, (, ),a,b\}$ of caret types.
\end{theorem}

This theorem is proved in the subsequent sections.  Moreover, we show that the language of normal forms defined over this symbol alphabet can never be accepted by a $1$-counter automaton.

\subsection{Word Acceptor}

The language $\F$ is constructed in stages, always taking the symbol alphabet to consist of the caret labels $\{e,r,(,),a,b\}$.

\begin{enumerate}
\item We first define a $1$-counter language $\LI$ containing all strings that uniquely describe possible interior subtrees of a finite rooted binary tree. Exterior caret labels are not used.
\item Next we extend $\LI$ to a 1-counter language $\LT$  whose strings uniquely describe finite rooted binary trees.
\item We define a $2$-counter language $L_T$ which is the intersection of $\otimes(\LT,\LT)$ with a regular language ensuring that no $\diamond$ symbols appear in any convolutions.  This language describes all pairs of finite rooted binary trees with the same number of carets.
\item Finally, we check that strings in $\LTT$ correspond to reduced pairs of binary trees by intersecting with a regular language $R$ which checks a simple reduction criterion. Then $\F = \LTT \cap R$.
\end{enumerate}

\subsubsection{A language describing interior subtrees}\label{SSS:InteriorSubtrees}
In \cite{Murray} a simple criterion is given for when a string of caret labels consisting of $``N'',``I'',``n''$ and $``i''$ (as defined in Figure \ref{F:Labels1}) represents an interior subtree of a finite rooted binary tree.  We restate this lemma using the alphabet $\SI = \{(,),a,b\}$ of interior caret labels. A translation between the two sets of labels is given in Section \ref{sec:carettypes}.

\begin{lemma} [Lemma 6 of \cite{Murray}] \label{L:Murray}
Let $ w = s_1 \cdots s_n$ be a word in $(\SI)^*$. Let $w_j$ denote the prefix $s_1 \cdots s_j$ and let $|u|_y$ denote the number of occurrences of $``y''$ in $u$. Then $w$ encodes a (possibly empty) interior subtree if and only if the following conditions hold:
\begin{enumerate}
\item For any $0 \leq j < n$ such that $|w_j|_( = |w_j|_)$, $s_{j+1} \in \{(,a\}$;
\item $|w|_( = |w|_)$.
\end{enumerate}
\end{lemma}

Let $\LI$ denote the set of all strings satisfying the conditions in Lemma \ref{L:Murray} (including the empty string). Then $\LI$ is a nonblind deterministic $1$-counter language accepted by a machine $\texttt{M}_{int}$, where the input is read from left to right and the counter instructions are as follows:
\begin{enumerate}
\item the counter increments by $1$ after reading $``(''$;
\item read $``)''$ only when the value of the counter is non-zero, and subsequently decrement the counter by $1$;
\item proceed to a fail state if the counter is zero and $``b''$ is read.
\end{enumerate}

These counter instructions correspond exactly to the following arrangement of carets in a binary tree.  A caret with label $``(''$ corresponds to the root caret of an interior subtree with a nonempty right subtree; there will be a corresponding caret labeled $``)''$ which is the final caret (in infix order) of this subtree.

As every element in $\LI$ corresponds to a possible interior subtree of a finite rooted binary tree, we can formalize this identification by defining  a mapping $\tau_{int}: \LI \to \{\textrm{interior subtrees}\}.$  To show that $\tau_{int}$ is well-defined we must verify that any word $w \in \LI$ encodes a unique interior subtree. This is easily checked by  considering all combinations of pairs of caret labels in $\SI = \{(,),a,b\}$ and verifying that each combination precisely defines where the corresponding carets must be positioned in the interior subtree.  A chart verifying this is included in the Appendix.  If $\alpha \in L_{int}$ is a string of caret types from $\SI$, we build a unique interior subtree as follows.  Each adjacent pair of caret types beginning on the left side of the string uniquely determines the relative placement of the carets in the subtree being constructed as shown in Figure \ref{F:InteriorCaretPlacement}.  Hence $\tau_{int}$ is a bijection between the set of possible interior subtrees of a finite rooted binary tree and the set of words in the language $\LI$.

\subsubsection{A language describing binary trees}

Next we construct a language $\LT$ over the alphabet $$\ST = \{e,r,(,),a,b\}$$ which accepts exactly those words corresponding to finite rooted binary trees. Let $\texttt{M}_{tree}$ denote the machine which accepts $\LT$.  A binary tree containing no interior carets will correspond to a word of the form $e^{n_1}re^{n_2}$, where $n_1,n_2 \in \N \cup \{0\}$.  In general, a binary tree has interior subtrees attached to any collection of exterior carets.  Thus, a string $w$ in $\SI^*$ describing such a tree must satisfy the following  conditions:
\begin{enumerate}
\item $w$ begins with either $``e''$ or $``r''$,
\item $w$ ends with either $``e''$ or $``r''$,
\item $w$ contains exactly one $``r''$, and
\item for any substring $s_{i-1} s_i \cdots s_j s_{j+1}$, if $s_{i-1}, s_{j+1} \in \{e,r\}$ and $s_i \cdots s_j \in \SI^*$, then $s_i \cdots s_j$ is a word in $\LI$.
\end{enumerate}
Let $\LT$ be the set of words satisfying conditions (1) through (4). Since $\LI$ is a 1-counter language, condition (4) can be checked with a single counter, as follows. Note that substrings $s_{i-1} s_i \cdots s_j s_{j+1}$ and $s_{k-1} s_k \cdots s_l s_{l+1}$ may overlap only at the first and final letters, which lie in $\{e,r\}$. Upon reading $``e''$ or $``r''$ with a subsequent letter from $\Sigma_{int}$, a machine scanning the word simulates the machine for $L_{int}$ on the substring $s_i \cdots s_j s_{j}$, and repeats this process when it encounters the next letter from $\{e,r\}$.

The three remaining conditions are easily be recognized with finite state automata. Hence $\LT$ is a nonblind deterministic $1$-counter language.

We can now extend the function $\tau_{int}: \LI \to \{\textrm{interior subtrees}\}$ to a bijection
$$\tau_{tree} : \LT \to \{\textrm{finite rooted binary trees}\}.$$

\subsubsection{A language describing pairs of binary trees}
To describe pairs of binary trees with the same number of carets, we use the convolution of two strings from $\LT$ which are the same length.  Namely, $\LTT$ is the intersection of $\otimes(\LT,\LT)$  with a regular language of convolutions of elements of $\ST^*$ containing no $\diamond$ symbols.  It follows from Lemma \ref{lemma:convolution} that $\LTT$ is a $2$-counter language.  Analogously, there is a bijection $$\tau: \LTT \rightarrow \{\textrm{pairs of finite rooted binary trees}\}.$$ If $\otimes(u,v) \in \LTT$ then the top line in the convolution determines the first tree in the pair, and the bottom line the second tree in the pair.  Let $M_{T}$ denote the machine which accepts $\LTT$.

\subsubsection{Reduction Criterion} Next we test whether a string in $\LTT$ corresponds to a reduced pair of trees.  We first observe sequences of labels which must appear in a string corresponding to a tree pair diagram that is unreduced. For each sequence, we intersect $\LTT$ with the complement of the corresponding regular language, using the fact that the set of regular languages is closed under complementation.

Given a finite rooted binary tree $T$, it is easily checked that:

\medskip

\begin{enumerate}
\item the first exterior caret of $T$ is exposed iff the word encoding it begins with $``ee''$ or $``er''$;
\item the last exterior caret of $T$ is exposed iff the word encoding it ends with $``re''$ or $``ee''$;
\item a symbol $``y''$ encodes an exposed interior caret of $T$ iff $y \in \{),a\}$ and is preceded by some  $x \in \{e,r,(,b\}$.  Recall that $``)''$ and $``a''$ are the two interior caret types which label carets that have no right child.
\end{enumerate}

There are three ways that a pair of binary trees representing a non-identity element of $F$ may be unreduced:
\begin{enumerate}
\item the first exterior caret of each tree is exposed; this occurs iff the word in $\LTT$ describing the pair of trees begins with ${e \choose e}{y_1 \choose y_2}$  for $y_i \in \{e,r\}$.
\item the last exterior caret of each tree is exposed; this occurs iff the word in $\LTT$ describing the pair of trees ends with: ${ x_1 \choose x_2} {e \choose e}$ for $x_i \in \{e,r\}$.
\item each tree has an exposed interior caret with infix number $n$; this occurs iff the word in $\LTT$ describing the pair of trees contains any of the following strings: ${ x_1 \choose x_2}{ y_1 \choose y_2}$ for $x_i \in \{e,r,(,b \}$ and $y_i \in \{),a \}$, where the $y_i$ represent the exposed interior carets.
\end{enumerate}

Each of the above cases can be detected with an finite state machine. Let $R$ denote the set of all strings that {\it do not} contain any of these sequences. Then $R$ is regular and, moreover, $R$ consists of exactly those strings that do not fail any of the three reduction criteria.  Let $\F = \LTT \cap R$. Since each word in $\LTT$ encodes a pair of binary trees with the same number of carets, each word in $\F$ encodes a reduced pair of binary trees, that is, a unique element of $F$. Furthermore, because $\LTT$ is a $2$-counter language, it follows that $\F$ is as well.

We can now restrict the bijection $\tau$ to obtain another bijection $$\tau_{red}: \F \rightarrow \{\textrm{reduced pairs of finite rooted binary trees}\}.$$  Thus we have shown that $\F$ is a $2$-counter language which is a normal form for Thompson's group $F$.

Let $\nu = \tau^{-1}$.  If $(T,S)$ is a pair of binary trees, then $\nu(T,S)$ is the string in $\LTT$ representing the caret types from the pair.  We abuse this notation in several ways.  If $g = (T,S)$, we write $\nu(g)$ interchangeably with $\nu(T,S)$.  If $g=(T,S)$ is reduced, then we understand that $\nu(g)=\tau^{-1}(g) = \tau_{red}^{-1}(g)$.  Thus in this case $\nu(g) \in \F$.

We now show that the language $\F$ is a quasigeodesic normal form for $F$.

\begin{proposition}\label{prop:fordham}
The normal form language $\F$ for Thompson's group $F$ is quasigeodesic.
\end{proposition}

\begin{proof}
It follows from work of Fordham \cite{F} that the number of carets in either tree in a reduced tree pair diagram for $g \in F$ is coarsely equivalent to the word length of $g$ with respect to the finite generating set $\{x_0,x_1\}$.  The length of an element of $\F$ is exactly the number of carets in either tree in the reduced tree pair diagram representing $g \in F$, hence  ${\mathcal F}$ is a quasigeodesic normal form.
\end{proof}

We next show that the language ${\mathcal F}$ of normal forms is not a context-free language.  Since $1$-counter languages form a subset of the set of context-free languages it follows that we can not improve on Theorem \ref{T:GCS}.  We prove Lemma \ref{L:NotCF} using Ogden's Lemma, which is a corollary to the standard pumping lemma for context-free languages, and is stated below.

\begin{lemma}[Ogden's Lemma, \cite{HU} lemma 6.2]\label{L:Ogden}
Let $L$ be a context free language.  Then there is a constant $p$ so that if $z$ is any word in $L$, and we mark any $p$ or more positions of $z$ as {\em distinguished}, then we can write $z=uvwxy$ so that
\begin{enumerate}
\item $v$ and $x$ together have at least one distinguished position,
\item vwx has at most $p$ distinguished positions, and
\item for all $i \geq 0$ we have $uv^iwx^iy \in L$.
\end{enumerate}
\end{lemma}
Ogden's Lemma allows us to omit the first letter of any string in the following proof from the substrings $v$ and $x$ in the application of the lemma, an assumption one cannot make when applying the standard pumping lemma for context-free languages.
\begin{lemma} \label{L:NotCF}
Let $\F$ be as above.  Then $\F$ is not context-free.
\end{lemma}
\begin{proof}
Assume for contradiction that $\F$ is context-free. Let $p$ be the constant guaranteed by Ogden's Lemma, and consider the convolution
$$w= {e \choose e} { ( \choose a }^p{ ( \choose ( }{ a \choose ( }^p{ ) \choose a }^p{ ) \choose ) }{ a \choose ) }^p{ r \choose r }$$
which lies in $\F$.  Mark all but the first and last symbol as distinguished, and apply Ogden's Lemma.  As $v$ and $x$ together must have at least one distinguished position, one of these strings  must contain either $``(''$ or $``)''$ in either the top or bottom line. First suppose $vx$ contains $``(''$ in the top line. Then $vx$ contains a symbol from the substring
$${ ( \choose a }^p{ ( \choose ( }.$$
However, since $|vwx| \leq p$ and $|{a \choose (}^p| = p$, it follows that $vx$ cannot contain the symbol $``)''$ in the top line. An almost identical argument shows that if $vx$ contains $``(''$ in the bottom line, it cannot also contain $``)''$ in the bottom line. In either case, $uv^iwx^iy \notin \F$ for $i \neq 1$, and hence $\F$ is not context free.
\end{proof}

\section{Multiplier languages for the $3$-counter graph automatic structure}\label{subsec:x_0multiplier}

We now construct the automata $M_{x_0}$ and $M_{x_1}$, which accept the multiplier languages $$L_{x_0} = \{\otimes(u,v) | u,v \in {\mathcal F}, \ \ou x_0 = \ov \}$$ and $$L_{x_1} = \{\otimes(u,v) | u,v \in {\mathcal F}, \ \ou x_1 = \ov \}.$$  As noted in \cite{ET}, when constructing ${\mathcal C}$-graph automatic groups, a multiplier automaton which reads the convolution $\otimes(u,v)$ must both check that $u$ and $v$ lie in the normal form language {\em and} that $\ou x_1 = \ov$.  This is not an issue for either automatic or graph automatic groups; in both of these scenarios, the normal form language is regular, and since intersections of regular languages are regular it is not necessary for the multiplier automata to check membership in the normal form language.  As our normal form languages are not regular, the multiplier automata must also verify this membership condition.  The language $L_{x_0}$ is accepted by a nonblind deterministic $2$-counter automaton, while $L_{x_1}$ is accepted by a nonblind deterministic $3$-counter automaton.

It is proven in \cite{ET} that for $\mathscr R=\{$regular languages$\}$ or $\mathscr C_1=\{$1-counter languages$\}$ we have $L_{x}\in \mathscr R$  (resp. $\mathscr C_1$) if and only if $L_{x^{-1}}\in \mathscr R$ (resp. $\mathscr C_1$).  Hence it suffices to consider only the multiplier languages  $L_{x_0}$ and $L_{x_1}$.

\subsection{The multiplier language $L_{x_0}$.} Multiplication of $g=(T,S)$ by $x_0$ induces a particular rearrangement of the subtrees of the original diagram, as shown in Figure~\ref{F:Rotation} when the root caret of $T$ has at least one left child.  To compute this product, the tree pair diagram for $x_0$, which corresponds to $\otimes(re,er) \in \F$, is placed to the left of the tree pair diagram from $g$.  Then (possibly) unreduced representatives of the two elements are created in which the middle trees are identical; this may entail adding an additional caret as a left child to caret $0$ in both $T$ and $S$.  The outermost trees together form the product $gx_0$.  The final result of any multiplication of tree pair diagrams is a unique reduced tree pair diagram, but this may include an intermediate stage where the resulting pair of trees is not reduced.

\begin{figure}[h!]
\begin{center}
\begin{tikzpicture}[scale=.7]

		\coordinate (A) at (1,0) {};
		\coordinate (B) at (0,-2) {};
		\node [state, rectangle, minimum size = 17pt] (C) at (2,-2) {$C$};
		\node[state, rectangle, minimum size = 17pt] (D) at (1,-4) {$B$};
		\node[state, rectangle, minimum size = 17pt] (E) at (-1,-4) {$A$};
		\node (F) at (1,.5) {$T$ from $g=(T,S)$};
		
		\node (H) at (2.9, -2.05) {};
		\node (I) at (4.1, -2.05) {};
		
		\path[->, thick] (H) edge (I);

    \foreach \point in {A,B}
        \fill [black] (\point) circle (1.5pt);

     	\draw (B) -- (A) -- (C);
	\draw (D) -- (B) -- (E);
	
		\coordinate (J) at (6,0) {};
		\node[state, rectangle, minimum size = 17pt] (K) at (5,-2) {$A$};
		\coordinate (L) at (7,-2) ;
		\node[state, rectangle, minimum size = 17pt] (M) at (6,-4) {$B$};
		\node[state, rectangle, minimum size = 17pt] (N) at (8,-4) {$C$};
		\node (O) at (6,.5) {$T'$ from $gx_0=(T',S')$};
		
     \foreach \point in {J,L}
        \fill [black] (\point) circle (1.5pt);

       \draw (K) -- (J) -- (L);
       \draw (M) -- (L) -- (N);

\end{tikzpicture}
\caption{Part of the tree pair diagrams for $g = (T,S)$ and $gx_0=(T',S')$.  The letters $A,B$, and $C$ represent possibly empty subtrees, and $S=S'$. }\label{F:Rotation}
\end{center}
\end{figure}
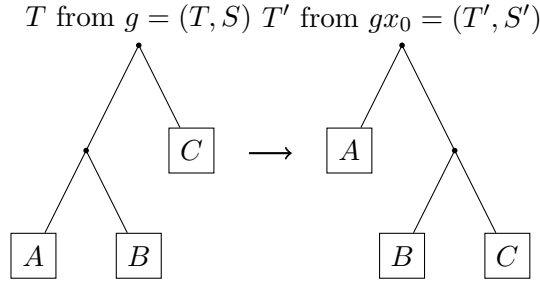

\begin{proposition}\label{prop:x0-mult}
The multiplier language for the generator $x_0$ defined by $$L_{x_0} = \{ \otimes(u,v) | u,v \in {\mathcal F} \ and \ \ou x_0=_F\ov\}$$ is recognized by a nonblind deterministic $2$-counter automaton.
\end{proposition}

\begin{proof}
The language $L_{x_0}$ will be the union of three sub-languages, based on the analysis of multiplication by $x_0$ given below.

Let $g = (T,S)$ as a reduced pair of trees.  Recall that $\nu(g) \in \F$ is the convolution of strings of caret types describing the trees $S$ and $T$.  To describe $u=\nu(g)$ and $\nu(gx_0)$ we divide into three cases depending on whether the root caret of $T$ has no left child, no right child, or has children on both sides.  In terms of $u=\nu(g)$, these cases correspond, respectively, to

\medskip

\begin{enumerate}
\item an initial letter of $``r''$ in the top line of $u$ when written as a convolution;
\item initial letter of $``e''$ and final letter $``r''$ in the top line of $u$ when written as a convolution;
\item neither of the above conditions.
\end{enumerate}
In each case we obtain a unique string in ${\mathcal F}$ which corresponds to $gx_0$.

Case 1: the root caret of $T$ has no left child. To perform the multiplication by $x_0$ a caret must be added to the left leaf of the root caret in $T$ and to the left leaf of caret $0$ in $S$, creating an unreduced representative $g^* = (T^*, S^*)$.  Let $gx_0 = (T', S')$ denote the product as a reduced tree pair diagram.

In this case we must have $\nu(g) = {r \choose z_0} { \gamma_1 \choose \gamma_2}$ for some $z_0 \in \{e,r\}$ and $\gamma_1,\gamma_2 \in \ST^*$, and using the combinatorial rearrangement of carets induced by multiplication by $x_0$, we can write

\medskip

\begin{itemize}
	\item $\nu(g) = {r \choose z_0} { \gamma_1 \choose \gamma_2}$,

\smallskip

	\item $\nu(g^*) = {e \choose e} {r \choose z_0} { \gamma_1 \choose \gamma_2}$, and

\smallskip

	\item $\nu(gx_0) = {r \choose e}{e \choose z_0} { \gamma_1 \choose \gamma_2}$.
\end{itemize}

Notice that this includes $g$ being the identity, in which case $\nu(g) = {r \choose r}$ and $\nu(gx_0) = {r \choose e} {e \choose r}$.

Case 2: the root caret of $T$ does have a left child and the following three conditions hold

\medskip

\begin{enumerate}
	\item[(i)] the root caret of $T$ has no right child,
	\item[(ii)] the left child of the root caret of $T$ has no right child, and
	\item[(iii)] the last exterior caret of $S$ is exposed.
\end{enumerate}
Then multiplication by $x_0$ yields an unreduced representative $(gx_0)^* = (T^*, S^*)$. We must eliminate a redundant caret from both $T^*$ and $S^*$ to obtain $gx_0 = (T', S')$.

Given conditions (i), (ii) and (iii), the string $\nu(g)$ must end with the symbols ${e \choose z_0} {r \choose e}$ for some $z_0 \in \{ e,r\}$.  Hence the string representing the unreduced $(gx_0)^* = (T^*, S^*)$ must ends in ${r \choose z_0} {e \choose e}$. To obtain the reduced tree pair diagram for the product, the final redundant caret must be eliminated from the tree pair diagram $(T^*, S^*)$; specifically, the right child of the root in $T^*$ has two exposed leaves, as does the final (rightmost) caret in $S^*$. These are removed to obtain the reduced tree pair diagram $(T',S')$ for $gx_0$.

Thus, for some $\gamma_1,\gamma_2 \in \Sigma_T^*$ and $z_0 \in \{e,r\}$ we have

\medskip

\begin{itemize}
	\item $\nu(g) = {\gamma_1 \choose \gamma_2}{e \choose z_0} { r \choose e}$,

\smallskip

	\item $\nu(gx_0^*) = {\gamma_1 \choose \gamma_2}{r \choose z_0} {e \choose e}$, and

\smallskip

	\item $\nu(gx_0) = {\gamma_1 \choose \gamma_2}{r \choose z_0} $.
\end{itemize}

Case 3: the element $g=(T,S)$ does not belong to Case 1 or Case 2.

In this case, the rearrangement of the subtrees of the domain tree is exactly as depicted in Figure~\ref{F:Rotation}. Multiplication has no effect on the range tree: that is, $S = S'$. Let $\gamma_A, \gamma_B, \gamma_C \in \LI$ denote the substrings that encode the subtrees $A, B,$ and $C$, respectively and $\chi_A,\chi_B,\chi_C \in \Sigma_T^*$ the strings of labels for the carets in $S$ paired, respectively, with the carets in subtrees $A, B$ and $C$ of $T$. Then
\begin{itemize}
	\item $\nu(g) = {\gamma_A \choose \chi_A}{e \choose z_0}{\gamma_b \choose \chi_B}{r \choose z_1}{\gamma_C \choose \chi_C} $ and

\smallskip

	\item $\nu(gx_0) = {\gamma_A \choose \chi_A}{r \choose z_0}{\gamma_b \choose \chi_B}{e \choose z_1}{\gamma_C \choose \chi_C}  $
\end{itemize}
for some $z_0,z_1 \in \Sigma_T$.

Define a language ${\mathcal N}_0$ to consist of the union of all convolutions $\otimes(u,v)$ of the following forms:

\medskip

\begin{itemize}
\item $u = {r \choose z_0} { \gamma_1 \choose \gamma_2}$ and $v={r \choose e}{e \choose z_0} { \gamma_1 \choose \gamma_2}$

\item $u= {\gamma_1 \choose \gamma_2}{e \choose z_0} { r \choose e} $ and $v= {\gamma_1 \choose \gamma_2}{r \choose z_0} $

\item $u= {\gamma_A \choose \chi_A}{e \choose z_0}{\gamma_b \choose \chi_B}{r \choose z_1}{\gamma_C \choose \chi_C} $ and $v= {\gamma_A \choose \chi_A}{r \choose z_0}{\gamma_b \choose \chi_B}{e \choose z_1}{\gamma_C \choose \chi_C} $
\end{itemize}
where $z_0,z_1 \in \Sigma_T$ and $\gamma_i \in \Sigma^*_T$.  Note that $u$ and $v$ may no longer lie in ${\mathcal F}$ or even $\LTT$.

It follows from Lemma \ref{L:RegularLanguages} that these three sub-languages are regular, and hence their union ${\mathcal N}_0$ is regular as well. The language $L_{x_0}$ is then the intersection of ${\mathcal N}_0$ with the $2$-counter language $\{\otimes(u,v) | u \in {\mathcal F}, \ v \in \Sigma_T^*\}$, because the construction of ${\mathcal N}_0$ guarantees that if $u \in {\mathcal F}$ and $\otimes(u,v) \in {\mathcal N}_0$, then the string $v$ lies in ${\mathcal F}$.  Hence the language $L_{x_0}$ can be recognized by a nonblind deterministic $2$-counter automaton.
\end{proof}

\subsection{The multiplier language $L_{x_1}$.}\label{subsec:x_1multiplier}
If $g = (T,S) \in F$, then the process of multiplying $g$ by $x_1$ is described in five cases which depend on the arrangement of carets in $T$ and $S$.

In this section, if $g = (T,S)$ is as in Figure \ref{F:Rotation2}, we will use $\gamma_A, \gamma_B,\gamma_C$ and $ \gamma_D$ to denote the strings in $L_{int}$ that encode the subtrees $A,B,C$ and $D$, respectively.  Let $\chi_A, \chi_B,\chi_C$ and $\chi_D$ denote the strings in $\Sigma_T^*$ that encode the carets of $S$ which are paired, respectively, with subtrees $A,B,C$ and $D$.  If $\gamma_R$ is empty for $R \in \{A,B,C,D\}$ then we say that $\chi_R$ is empty as well.

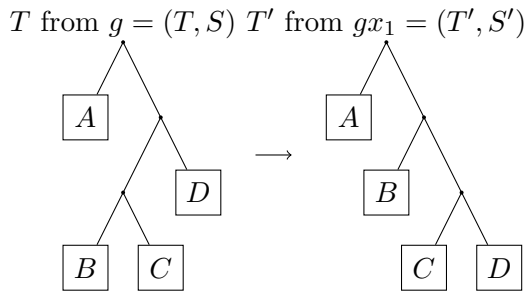
\begin{figure}[h!]
\begin{center}
\begin{tikzpicture}[scale=.5]

		\coordinate (A) at (1,0) {};
		\coordinate (B) at (0,-2) {};
		\node [state, rectangle, minimum size = 17pt] (C) at (2,-2) {$D$};
		\node[state, rectangle, minimum size = 17pt] (D) at (1,-4) {$C$};
		\node[state, rectangle, minimum size = 17pt] (E) at (-1,-4) {$B$};
		\node (F) at (0,2.5) {$T$ from $g=(T,S)$};
		\coordinate (X) at (0,2);
		\node[state, rectangle, minimum size = 17pt] (Y) at (-1,0) {$A$};
		
		\node (H) at (3.25, -1) {};
		\node (I) at (4.75, -1) {};
		\path[->] (H) edge (I);
		\draw (Y) -- (X) -- (A);

    \foreach \point in {A,B,X}
        \fill [black] (\point) circle (1.5pt);

     	\draw (B) -- (A) -- (C);
	\draw (D) -- (B) -- (E);
	
		\coordinate (J) at (8,0) {};
		\node[state, rectangle, minimum size = 17pt] (K) at (7,-2) {$B$};
		\coordinate (L) at (9,-2) {};
		\node[state, rectangle, minimum size = 17pt] (M) at (8,-4) {$C$};
		\node[state, rectangle, minimum size = 17pt] (N) at (10,-4) {$D$};
		\node (O) at (7,2.5) {$T'$ from $gx_1=(T',S')$};
		\node[state, rectangle, minimum size = 17pt] (Z) at (6,0) {$A$};
		\coordinate (W) at (7,2);
		
     \foreach \point in {J,L,W}
        \fill [black] (\point) circle (1.5pt);

       \draw (K) -- (J) -- (L);
       \draw (M) -- (L) -- (N);
       \draw (J) -- (W);
       \draw (Z) -- (W);

\end{tikzpicture}
\caption{If $g = (T,S)$ and $gx_1=(T',S')$, where the tree $T$ contains at least a root caret whose right child has a left child, the combinatorial rearrangement of the subtrees of $T$ yielding $T'$.  Here, $A,B, C$ and $D$ represent possibly empty subtrees and $S=S'$.}\label{F:Rotation2}
\end{center}
\end{figure}

We now identify all pairs of strings in $\F$ which correspond to possible pairs $g$ and $gx_1$.  In Cases (1) through (4) below the difference between $\nu(g)$ and $\nu(gx_1)$ lies either at the end of the string or immediately after the unique caret of type $``r''$ in the top line of the convolution $\nu(g)$.  A finite state machine can easily detect this change.  Then two counters are required to ensure that $\nu(g)$ lies in $\F$, and by the construction of the language, $\nu(gx_1)$ must lie in $\F$ as well. Thus the set of all convolutions described in Cases (1) through (4) will form a nonblind deterministic $2$-counter language. The set of strings described in the final case of this section will require an additional counter to detect that the change between $\nu(g)$ and $\nu(gx_1)$ lies at the correct position in the interior of the convolution.

Case 1: Two carets must be added to $g=(T,S)$ in order to multiply by $x_1$.  Suppose first that the root caret of $T$ has no right child. Then we must add a right child with a left child to the rightmost carets of $T$ and $S$, creating an unreduced representative of $g$, before performing the multiplication. Let $g^* = (T^* , S^*)$ denote this unreduced representative for $g$.

Since the root caret of $T$ has no right child, the string corresponding to $T$ ends with $``r''$. Then for some $z_0 \in \{e,r\}$, we have

\medskip

\begin{itemize}
	\item $\nu(g) = { \gamma_A \choose \chi_A} {r \choose z_0}$,
	\item $\nu(g*) = {\gamma_A \choose \chi_A} {r \choose z_0} {a \choose a} {e \choose e}$, and
	\item $\nu(gx_1) = {\gamma_A \choose \chi_A} {r \choose z_0} {e \choose a} {e \choose e}$.
\end{itemize}
Notice that this case also includes the identity with $\gamma_A = \chi_A = \epsilon$ and $z_0 = r$. Then $\nu(x_1) = {r \choose r} {e \choose a} {e \choose e}$.

Case 2: One caret must be added to $g=(T,S)$ in order to multiply by $x_1$.  Suppose that the root caret of $T$ has a right child which does not have a left child, so that the string in $\Sigma_T^*$ corresponding to $T$ is $\gamma_A r e \gamma_D$. To perform  multiplication by $x_1$, we must add a left child to the right child of the root caret of $T$ and to the corresponding leaf in $S$, creating an unreduced representative of $g$, which we denote $g^* = (T^* , S^*)$.  We consider three possibilities for this multiplication, based on the caret type $z_0$ of the caret in $S$ paired with the root caret of $T$.  Let $n$ denote the infix number of the root caret in $T$ and the corresponding caret of type $z_0$ in $S$.

In subcases (a)-(c) below, we always have
$$\nu(g) = {\gamma_A \choose \chi_A} {r \choose z_0} {e \choose z_1} {\gamma_D \choose \chi_D}$$
where $z_0,z_1 \in \Sigma_T$.

\begin{enumerate}
	\item[(a)] First suppose $z_0 \in \{e,r,(,b)\}$. Then the new caret is added to $S$ as the left child of the caret of infix number $n+1$ and will be type $``a''$ in $S^* = S'$. Then we have

\smallskip

	\begin{itemize}
		\item $\nu(g^*) = {\gamma_A \choose \chi_A} {r \choose z_0} {a \choose a} {e \choose z_1} {\gamma_D \choose \chi_D}$, and
		\item $\nu(gx_1) =  {\gamma_A \choose \chi_A} {r \choose z_0} {e \choose a} {e \choose z_1} {\gamma_D \choose \chi_D}$.
	\end{itemize}

\smallskip

	\item[(b)] Next suppose $z_0$ is type $``a''$. Then the new caret is added to $S$ as the right child of caret $n$ (of type $z_0$), which changes they type of caret $n$ in $S$ to $``(''$, and the new caret is of type $``)''$. Then we have

\smallskip

	\begin{itemize}
		\item $\nu(g^*) = {\gamma_A \choose \chi_A} {r \choose (} {a \choose )} {e \choose z_1} {\gamma_D \choose \chi_D}$, and
		\item $\nu(gx_1) =  {\gamma_A \choose \chi_A} {r \choose (} {e \choose )} {e \choose z_1} {\gamma_D \choose \chi_D}$.
	\end{itemize}

\smallskip

	\item[(c)] Finally, suppose $z_0$ is type $``)''$. As in (b), the new caret is added to $S$ as the right child of caret $n$, which changes to type $``b''$, and the new caret is of type $``)''$. Then we have:

\smallskip

	\begin{itemize}
		\item $\nu(g^*) = {\gamma_A \choose \chi_A} {r \choose b} {a \choose )} {e \choose z_1} {\gamma_D \choose \chi_D}$, and
		\item $\nu(gx_1) =  {\gamma_A \choose \chi_A} {r \choose b} {e \choose )} {e \choose z_1} {\gamma_D \choose \chi_D}$.
	\end{itemize}

\end{enumerate}
In all three cases, $\nu(gx_1)$ is a reduced string.  If $\chi_D$ is empty, then $z_1 \neq e$ because $\nu(g)$ is itself a reduced string.  Note that for a given string $\nu(g)$ in Case 2, a finite state automaton can decide which subcase $\nu(g)$ lies in.

Case 3: After multiplication by $x_1$, two carets must be removed to obtain the reduced tree pair diagram for the product. Suppose that the root caret of $T$ has a right child which has a left child. If the following three conditions hold
\begin{enumerate}
\item[(i)] the right child of the root caret of $T$ does not have a right child;
\item[(ii)] the left child of the right child of the root caret of $T$ has no children; and
\item[(iii)] the final two carets of $S$ are exterior carets with no interior children
\end{enumerate}
then multiplication yields an unreduced representative $gx_1^* = (T^* , S^*)$. We must eliminate two redundant carets from each of $T^*$ and $S^*$ to obtain $gx_1 = (T', S')$.

Given conditions (i) and (ii), the string corresponding to $T$ ends with $``rae''$. Condition (iii) implies that the string corresponding to $S$ ends with $z_0ee$ for $z_0 \in \{r, e\}$. Then we have

\medskip

\begin{itemize}
	\item $\nu(g) = { \gamma_A \choose \chi_A} {r \choose z_0} {a \choose e} {e \choose e }$,
	\item $\nu(gx_1^*) = {\gamma_A \choose \chi_A} {r \choose z_0} {e \choose e } {e \choose e}$, and
	\item $\nu(gx_1) = {\gamma_A \choose \chi_A}  {r \choose z_0}$.
\end{itemize}

Case 4: After multiplication by $x_1$, one caret must be removed to obtain the reduced tree pair diagram for the product. Suppose that the string $w$ does not lie in Cases 1, 2, or 3. Suppose further that the root caret of $T$ has a right child that has a left child. If the following three conditions hold
\begin{enumerate}
\item[(i)] the right child of the root caret of $T$ does not have a right child;
\item[(ii)] the left child of the right child of the root caret of $T$ does not have a right child (but may have a left child); and
\item[(iii)] the final caret of $S$ is an exterior caret with no left child
\end{enumerate}
then multiplication yields an unreduced representative $gx_1^* = (T^*, S^*)$. We must eliminate a single redundant caret from each of $T^*$ and $S^*$ to obtain $gx_1 = (T', S')$.

It follows from conditions (i) and (ii) that the string in $\Sigma_T^*$ corresponding to $T$ is $\gamma_A r \gamma_B a e$. Then for $z_0\in \Sigma_T$ and $z_1 \in \{e,r\}$ we have

\medskip

\begin{itemize}
	\item $\nu(g) = {\gamma_A \choose \chi_A} {r \choose z_0} {\gamma_B \choose \chi_B} {a \choose e} {e \choose e}$,
	\item $\nu(gx_1^*) = {\gamma_A \choose \chi_A} {r \choose z_0} {\gamma_B \choose \chi_B} {e \choose e} {e \choose e}$, and
	\item $\nu(gx_1) = {\gamma_A \choose \chi_A} {r \choose z_0} {\gamma_B \choose \chi_B} {e \choose e}$.
\end{itemize}
To avoid overlap with Case (3), if $\chi_B = 0$ then $z_0 \notin \{e,r\}$ to ensure that there is only one redundant caret in $\nu(gx_1^*)$.

To summarize Cases (1) through (4), let ${\mathcal N_1}$ be the set of all $\otimes(u,v)$ where $u=\nu(g)$ and $v=\nu(gx_1)$ are as listed in Cases (1) through (4), where $\gamma_i$ and  $\chi_i$ are now allowed to be any strings in $\Sigma_T^*$.  As with the construction of $L_{x_0}$, we are not concerned with whether $u,v \in \F$.  Each possible difference between the strings $u$ and $v$ in the previous four cases is easily detected by a finite state machine, and thus the language ${\mathcal N}_1$ is regular.
Finally, intersect ${\mathcal N}_1$ with the language $\{\otimes(x,y) \mid x \in \F, \ y \in \Sigma_T^*\}$ where the latter is a nonblind deterministic $2$-counter language.  Denote the intersection by ${\mathcal N}$; it is also a nonblind deterministic $2$-counter language and consists of exactly those convolutions covered in Cases (1) through (4). Note that it is enough to check that one of $u$ and $v$ lies in $\F$ because the construction of ${\mathcal N}_1$ ensures that when one component lies in $\F$, the other one does as well.

Case 5: No carets must be added to multiply by $x_1$ and the resulting tree pair diagram is necessarily reduced.  That is, suppose that $g = (T,S)$ does not belong to any of the previous cases. Then multiplication by $x_1$ does not affect the range tree;  if $g = (T,S)$ and $gx_1 = (T', S')$, then $S = S'$. The trees $T$ and $T'$ are depicted in Figure \ref{F:Rotation2}.  We consider two cases according to whether the subtree $C$ is empty, and note that a finite state machine can detect which case a convolution lies in.   Note that $\otimes(\Sigma_T^*,\F)$ is a $2$-counter language, because $\Sigma_T^*$ is regular and $\F$ requires two counters.

In Case 5, there are at most two carets which change type between $T$ and $T'$:
\begin{enumerate}
\item The parent caret of subtrees $B$ and $C$; this caret is interior in $T$ and exterior in $T'$, and we denote it $d$.  This change always occurs.

\item The root caret of subtree $C$, which changes type if and only if $C$ is not empty.  We denote this caret by $f$.
\end{enumerate}
We show below that the $x_1$ multiplier language can be easily expressed using $3$ counters.  In this case, we switch our perspective and focus on the tree $T'$  from the tree pair diagram for $gx_1$ rather than the tree $T$.

Case 5(a): Subtree $C$ is empty.  In this case,
$$u=\nu(g) =  {\gamma_A \choose \chi_A} {r \choose z_0} {\gamma_B \choose \chi_B} {a \choose z_1} {e \choose z_2} {\gamma_D \choose \chi_D}$$
 and
$$v=\nu(gx_1) =  {\gamma_A \choose \chi_A} {r \choose z_0} {\gamma_B \choose \chi_B} {e \choose z_1} {e \choose z_2} {\gamma_D \choose \chi_D}$$
for some $z_0,z_1,z_2 \in \Sigma_T$.
Let ${\mathcal K}_1$ be the regular language of all strings $\otimes(u,v)$ where $u$ and $v$ have the above form with $\gamma_B$ and $\gamma_C$ replaced by arbitrary strings from $\Sigma_{int}^*$ and $\gamma_A, \ \gamma_D$ and $\chi_Y$ replaced by arbitrary strings from $\Sigma_T^*$, for $Y \in \{A,B,D\}$.  As $\gamma_B$ consists only of interior caret types, a finite state machine can easily check that after ${r \choose z_0}$ is read in $v$, the next symbol ${ x \choose y}$ where $x \notin \{a,b,(,)\}$ is ${e \choose z_1}$ and that it is paired with ${ a \choose z_1}$ in $u$.  Assume without loss of generality that this condition is also verified in ${\mathcal K}_1$.  Then ${\mathcal L}_1={\mathcal K}_1 \cap \otimes(\F,\Sigma_T^*)$ is a $2$-counter language accepting exactly those convolutions $\otimes(u,v)$ in Case 5(a), that is, with $C= \emptyset$.

Case 5(b): Subtree $C$ is not empty.  In this case, in $\nu(gx_1)$ write $\gamma_C = \eta_1 f \eta_2$ and in $\nu(g)$ write $\gamma_C = \eta_1 \bar{f} \eta_2$ where $f$, resp. $\bar{f}$, corresponds to the root caret of the subtree $C$.  Here $\eta_i \in L_{int}^*$, and we let $\eta_1'$ and $\eta_2'$ in $\Sigma_T^*$ respectively denote the strings of caret types paired with the caret types in $\eta_1$ and $\eta_2$.  Then
$$u=\nu(g) =  {\gamma_A \choose \chi_A} {r \choose z_0} {\gamma_B \choose \chi_B} {( \choose z_1} {\eta_1 \choose \eta_1'} {\bar{f} \choose z_3} {\eta_2 \choose \eta_2'} {e \choose z_2} {\gamma_D \choose \chi_D}$$
 and
$$v=\nu(gx_1) =  {\gamma_A \choose \chi_A} {r \choose z_0} {\gamma_B \choose \chi_B} {e \choose z_1} {\eta_1 \choose \eta_1'} {f \choose z_3} {\eta_2 \choose \eta_2'} {e \choose z_2} {\gamma_D \choose \chi_D}$$
where $z_0,z_1,z_2,z_3 \in \Sigma_T$ and
\begin{enumerate}
\item if the right subtree of the root caret of $C$ is empty then $\bar{f} = )$ and $f=a$, or
\item if the right subtree of the root caret of $C$ is not empty then $\bar{f} = b$ and $f=($.
\end{enumerate}
Let ${\mathcal K}_2$ be the regular language of all strings $\otimes(u,v)$ where $u$ and $v$ have the above form with all $\gamma_B, \ \eta_1$ and $\eta_2$ replaced by arbitrary strings from $\Sigma_{int}^*$ and $\gamma_A, \gamma_D, \chi_Y, \eta_1'$ and $\eta_2'$ replaced by arbitrary strings from $\Sigma_T^*$, for $Y \in \{A,B,D\}$.  As in Case 5(a), assume that the finite state machine recognizing ${\mathcal K}_2$ verifies that the change in the caret type of the parent caret of subtree $B$ occurs in the appropriate position (that is, the first $``e''$ following the root caret in the string representing $T'$).

To detect that the change in type from caret $\overline{f}$ to caret $f$ occurs in the correct position in the string we will use a single counter. The counter instructions used to show that a string from $\Sigma_{int}^*$ forms a valid interior subtree were as follows:
\begin{enumerate}
\item When the machine reads $``(''$ the counter is incremented by $1$.
\item The machine may read $``)''$ only when the value of the counter is positive, and then the value of the counter is decremented by $1$.
\item If the value of the counter is $0$ and the letter $``b''$ is read, the machine proceeds to a fail state.
\end{enumerate}
When these counter instructions are applied to a machine reading $\tau^{-1}_{tree}(T')$, the value of the counter is $0$ before each substring $\gamma_Y$ is read, for $Y \in \{A,B,C,D\}$.
Moreover, after the final letter of $\eta_1$ is read, the value of the counter is again $0$.  If $\eta_2$ is nonempty, then upon reading $``(''$ corresponding to the root caret of subtree $C$ (caret $f$) the value of the counter is set to $1$ and will not equal zero again until after the entire string $\eta_2$ has been read by the machine.  Hence in the Figure \ref{fig:case5} we check that the value of the counter is strictly positive as $\eta_2$ is read.

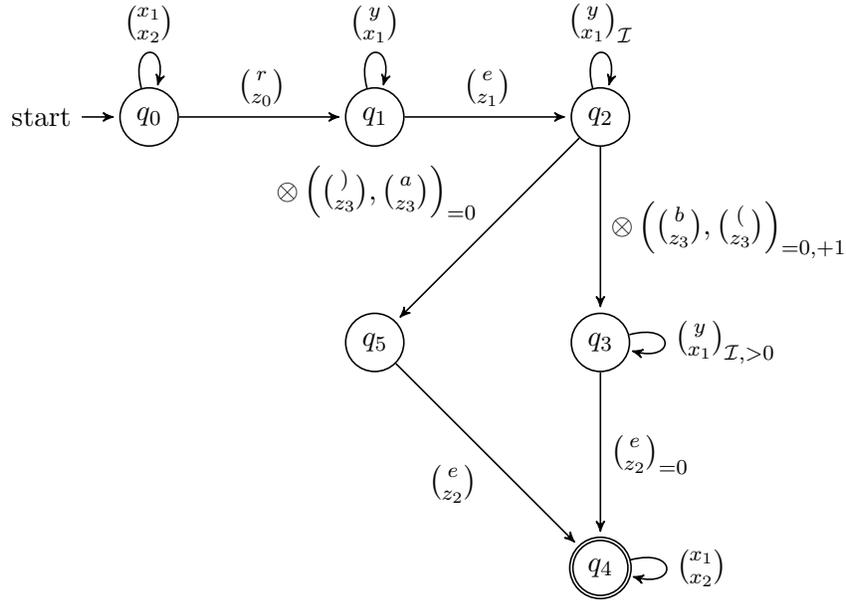
\begin{figure}[ht!]
\begin{center}
\tikzset{every state/.style={minimum size=2em}} 
\begin{tikzpicture}[->,>=stealth',shorten >=2pt,auto, node distance=3cm,
semithick]

\node[state, initial] (S) {$q_0$};
\node[state] (A) [right of=S] {$q_1$};
\node[state] (B) [right of=A] {$q_2$};
\node[state] (C) [below of=B]{$q_3$};
\node[state, accepting] (D) [below of=C] {$q_4$};
\node[state] (E) [below of=A] {$q_5$};

\path (S)
edge node {${r \choose z_0}$} (A)
edge [loop above] node {${x_1 \choose x_2}$} (S);

\path(A)
edge [loop above] node {${y \choose x_1}$} (A)
edge node {${e \choose z_1}$} (B);

\path(B)
edge [loop above] node {${y \choose x_1}_{\mathcal I}$} (B)
edge [swap] node {$\otimes\left({) \choose z_3},{a \choose z_3}\right)_{=0}$} (E)
edge node {$\otimes\left({b \choose z_3},{ ( \choose z_3}\right)_{=0,+1}$} (C);

\path(C)
edge [loop right] node {${y \choose x_1}_{{\mathcal I}, >0}$} (C)
edge node {${e \choose z_2}_{=0}$} (D);

\path(D)
edge [loop right] node {${x_1 \choose x_2}$} (D);

\path(E)
edge [swap] node {${e \choose z_2}$} (D);

\end{tikzpicture}
\caption{For $\otimes(u,v) $ as in Case 5, this machine checks that the change in the root caret of subtree $C$ from $T$ to $T'$ occurs in the correct spot in the string and has the correct label change.  The edge labels correspond to the letters in the string $v=\nu(gx_1) \in \F$ and the corresponding letter from $u=\nu(g)$ is ignored, except at the root caret of subtree $C$, where both the letter from $u$ and from $v$ are considered.  The letters $z_0,z_1,z_2$ are as in the expressions for $u$ and $v$ above, $y \in \Sigma_{int}$ and $x_i \in \Sigma_T$. The start state is $q_0$ and the accept state is $q_4$.
}
\label{fig:case5}
\end{center}
\end{figure}

The machine in Figure \ref{fig:case5} reads convolutions $\otimes(u,v) \in {\mathcal K}_2$.  The state transitions are determined solely by the letter from $v$, with one exception: the carets corresponding to the root caret of the subtree $C$ in $T'$ (or $T$).  In Figure \ref{fig:case5} we label the state transitions only by the letter read in $v$ except at this position.  When $v \in \F$ as well, this machine works as follows.  Let $y \in \Sigma_{int}$ and $x \in \Sigma_{T}$ be arbitrary letters.  Let ${\mathcal I}$ denote the set of counter instructions given above.  The label ${ \alpha \choose \beta}_{\mathcal I}$ denotes the appropriate counter rules from ${\mathcal I}$ which apply to the letter ${ \alpha \choose \beta}$, for $\alpha,\beta \in \Sigma_T$.

\begin{enumerate}
\item The loop at state $q_0$ reads the string $\gamma_A$ corresponding to subtree $A$ in $T'$ and the corresponding string $\chi_A$ from $\nu(gx_1)$.  The edge labeled ${ r \choose z_0}$ can only be read when this subtree is completed.

    \medskip

\item The loop at state $q_1$ reads the string $\gamma_B$ corresponding to subtree $B$ in $T'$ and the corresponding string $\chi_B$ from $\nu(gx_1)$. The edge labeled ${ e \choose z_1}$ can only be read when this subtree is completed.

    \medskip

\item The loop at state $q_2$ reads the string $\eta_1$ corresponding to the left subtree of the root caret of subtree $C$ in $T'$ and the corresponding string $\eta_1'$. We apply the counter instructions ${\mathcal I}$ to this part of the machine.  After this string has been read, the value of the counter is $0$.

    \medskip

\item If the right subtree of the root caret of subtree $C$ is empty, then the root caret of $C$ in $\nu(gx_1)$ has type $``)''$ in $T'$ and the machine transitions to state $q_5$ verifying the counter value.  As this corresponds to the root caret of $C$ the machine simultaneously checks that the corresponding symbol from $u=\nu(g)$ is ${ a \choose z_3}$ before transitioning to state $q_5$.

    \medskip

\item If the right subtree of the root caret of subtree $C$ in $\nu(gx_1)$ is not empty, then the root caret of $C$ has type $``(''$ in $T'$ and the machine transitions to state $q_3$ verifying the counter value.  The machine simultaneously checks that the corresponding symbol from $u=\nu(g)$ is ${ b \choose z_3}$. The loop at state $q_3$ reads the string $\eta_2$ corresponding to the right subtree of the root caret of $C$, using the counter instructions ${\mathcal I}$ but additionally verifying that the value of the counter is never $0$ while this string is read.  After the entire string $\eta_2$ is read the counter value is $0$, and this is verified along the edge leading to state $q_4$.
    \medskip

\item The loop at state $q_4$ reads the string $\gamma_D$ corresponding to subtree $D$ in $T'$ and the corresponding string $\chi_D$.
\end{enumerate}

Let ${\mathcal L}_2$ be the intersection of the language accepted by the machine in Figure \ref{fig:case5} with $\otimes(\Sigma_T^*,\F)$; the latter verifies that $v \in \F$ (and hence $u \in \F$) as well.  Then ${\mathcal L}_2$ is the language of all convolutions accepted in Case 5(b); as it is the intersection of a $1$-counter language with a $2$-counter language, we conclude that ${\mathcal L}_2$ is a $3$-counter language. It follows that $L_{x_1}$ is the union of ${\mathcal L}_1$ and ${\mathcal L}_2$, hence a nonblind deterministic $3$-counter language. This completes the proof of Theorem \ref{T:GCS}.

\section{appendix}

Figure \ref{F:InteriorCaretPlacement} below shows that the map $\tau_-'$ defined in Section \ref{Sec:caret-lang} is well defined.

\input{huge-table}

\bibliographystyle{plain}
\bibliography{refs}

\def\cprime{$'$}
\begin{thebibliography}{10}

\bibitem{BB}
James~M. Belk and Kenneth~S. Brown.
\newblock Forest diagrams for elements of {T}hompson's group {$F$}.
\newblock {\em Internat. J. Algebra Comput.}, 15(5-6):815--850, 2005.

\bibitem{BG}
Kenneth~S. Brown and Ross Geoghegan.
\newblock An infinite-dimensional torsion-free {${\rm FP}_{\infty }$} group.
\newblock {\em Invent. Math.}, 77(2):367--381, 1984.

\bibitem{CFP}
J.~W. Cannon, W.~J. Floyd, and W.~R. Parry.
\newblock Introductory notes on {R}ichard {T}hompson's groups.
\newblock {\em Enseign. Math. (2)}, 42(3-4):215--256, 1996.

\bibitem{CT1}
Sean Cleary and Jennifer Taback.
\newblock Combinatorial properties of {T}hompson's group {$F$}.
\newblock {\em Trans. Amer. Math. Soc.}, 356(7):2825--2849 (electronic), 2004.

\bibitem{CT}
Sean Cleary and Jennifer Taback.
\newblock Seesaw words in {T}hompson's group {$F$}.
\newblock In {\em Geometric methods in group theory}, volume 372 of {\em
  Contemp. Math.}, pages 147--159. Amer. Math. Soc., Providence, RI, 2005.

\bibitem{Murray}
Murray Elder, {\'E}ric Fusy, and Andrew Rechnitzer.
\newblock Counting elements and geodesics in {T}hompson's group {$F$}.
\newblock {\em J. Algebra}, 324(1):102--121, 2010.

\bibitem{ET1}
Murray Elder and Jennifer Taback.
\newblock Thompson's group {$F$} is 1-counter graph automatic.
\newblock Preprint, 2015.

\bibitem{ET}
Murray Elder and Jennifer Taback.
\newblock {${\mathcal C}$}-graph automatic groups.
\newblock {\em J. Algebra}, 413:289--319, 2014.

\bibitem{F}
S.~Blake Fordham.
\newblock Minimal length elements of {T}hompson's group {$F$}.
\newblock {\em Geom. Dedicata}, 99:179--220, 2003.

\bibitem{GS}
V.~S. Guba and M.~V. Sapir.
\newblock The {D}ehn function and a regular set of normal forms for {R}.
  {T}hompson's group {$F$}.
\newblock {\em J. Austral. Math. Soc. Ser. A}, 62(3):315--328, 1997.

\bibitem{HU}
John~E. Hopcroft and Jeffrey~D. Ullman.
\newblock {\em Introduction to automata theory, languages, and computation}.
\newblock Addison-Wesley Publishing Co., Reading, Mass., 1979.
\newblock Addison-Wesley Series in Computer Science.

\bibitem{KKM}
Olga Kharlampovich, Bakhadyr Khoussainov, and Alexei Miasnikov.
\newblock From automatic structures to automatic groups.
\newblock {\em Groups Geom. Dyn.}, 8(1):157--198, 2014.

\bibitem{SU}
Vladimir Shpilrain and Alexander Ushakov.
\newblock Thompson's group and public key cryptography.
\newblock In John Ioannidis, Angelos Keromytis, and Moti Yung, editors, {\em
  Applied Cryptography and Network Security}, volume 3531 of {\em Lecture Notes
  in Computer Science}, pages 151--163. Springer Berlin Heidelberg, 2005.

\end{thebibliography}

\end{document}